\documentclass{amsart}

\usepackage{amssymb,color}
\usepackage{amsfonts}
\usepackage{amsmath}
\usepackage{euscript}
\usepackage{enumerate}
\usepackage{graphics}
\usepackage{pdfsync}
\newtheorem{theorem}{Theorem}[section]
\newtheorem{lemma}[theorem]{Lemma}
\newtheorem{note}[theorem]{Note}
\newtheorem{prop}[theorem]{Proposition}
\newtheorem{cor}[theorem]{Corollary}

\newtheorem{rem}[theorem]{Remark}

\newtheorem*{Theorem1'}{Theorem 1'}

\theoremstyle{definition}

\theoremstyle{remark}

\numberwithin{equation}{section}

\newcommand \F{\mathbb{F}}

\begin{document}

\title[Artin-Schreier towers of finite fields]{Artin-Schreier towers of finite fields}

\author{Leandro Cagliero}
\address{CIEM-CONICET, FAMAF-Universidad Nacional de C\'ordoba, Argentina.}
\email{cagliero@famaf.unc.edu.ar}
\thanks{The first author was partially supported by grants from CONICET, FONCYT and SeCyT-UNC\'ordoba).}

\author{Allen Herman}
\address{Department of Mathematics and Statistics, Univeristy of Regina, Canada}
\email{Allen.Herman@uregina.ca}
\thanks{The second author was partially supported by NSERC grant 2023-00014}

\author{Fernando Szechtman}
\address{Department of Mathematics and Statistics, Univeristy of Regina, Canada}
\email{fernando.szechtman@gmail.com}
\thanks{The third author was partially supported by NSERC grant 2020-04062}



\begin{abstract} Given a prime number $p$, we consider the tower of finite fields $\F_p=L_{-1}\subset L_0\subset L_1\subset\cdots$,
where each step corresponds to an Artin-Schreier extension of degree $p$, so that for $i\geq 0$,
$L_{i}=L_{i-1}[c_{i}]$, where $c_i$ is a root of $X^p-X-a_{i-1}$ and $a_{i-1}=(c_{-1}\cdots c_{i-1})^{p-1}$, with $c_{-1}=1$.
We extend and strengthen to arbitrary primes prior
work of Popovych for $p=2$ on the multiplicative order $O(c_i)$ of the given generator $c_i$ for $L_i$ over $L_{i-1}$.
In particular, for $i\geq 0$, we show that $O(c_i)=O(a_i)$, except only when $p=2$ and $i=1$,
and that $O(c_i)$ is equal to the product of the orders of $c_j$
modulo $L_{j-1}^\times$, where $0\leq j\leq i$ if $p$ is odd, and $i\geq 2$ and $1\leq j\leq i$ if $p=2$.
We also show that for $i\geq 0$, the $\mathrm{Gal}(L_i/L_{i-1})$-conjugates of $a_i$ form a normal basis of $L_i$ over $L_{i-1}$.
In addition, we obtain the minimal polynomial of $c_1$ over $\F_p$ in explicit form.
\end{abstract}

\maketitle

\section{Introduction}

We fix throughout the paper a prime number $p$ and an algebraic closure $E$ of the field $\F_p$ with $p$ elements.
Recursively define $c_{-1},c_0,c_1,\dots\in E$ by taking $c_{-1}=1$
and selecting $c_{i+1}$ so that 
$$
c_{i+1}^p-c_{i+1}-(c_{-1}\cdots c_i)^{p-1}=0,\quad i\geq -1.
$$
We further set $L_{-1}=\F_p$, as well as
$$
L_{i+1}=L_i[c_{i+1}],\; a_i=(c_{-1}\cdots c_i)^{p-1},\quad i\geq -1,
$$
and
$$
N_i=\frac{p^{p^{i+1}}-1}{p^{p^i}-1},\quad i\geq 0.
$$
Then 
\begin{equation}\label{torref}
L_{-1}\subset L_0 \subset L_1 \subset L_2 \subset\cdots,
\end{equation}
is a tower of finite fields, where $[L_{i+1}:L_{i}]=p$ for all $i\geq -1$ and
\begin{equation}\label{sizeli}
|L_i^\times |=p^{p^{i+1}}-1=(p-1)N_0\cdots N_i,\quad i\geq -1.
\end{equation}

This construction was first given by Albert \cite{A} in 1934, following prior work by Artin and Schreier \cite{AS} in 1927.
More general algebraic extensions of finite fields are discussed in the book \cite{BS} by Brawley and Schnibben.
When $p=2$, the above towers are described in Conway's book \cite{C} as well as in Lenstra's notes \cite{L}.

The first extension $L_0/\F_p$ is connected to the period conjecture for the Bell numbers modulo $p$, which reduces to 
the question of whether or not the root $c_0$ of $X^p-X-1$ has multiplicative order $N_0$. This problem has been 
addressed, among others, by Wagstaff \cite{Wa}; Car, Gallardo, Rahavandrainy, and Vaserstein \cite{CGRV}; Montgomery, 
Nahm, and Wagstaff \cite{MNW}; and Gallardo in \cite{G1} and \cite{G2}.  While significant progress has been made, it 
remains open, as well as the question of whether $N_0$ is always square-free.

Albert's construction was rediscovered by Ito, Kashiwara, and Huling \cite{IKH}. They
give an algorithm to produce irreducible polynomials of high degree related to the tower (\ref{torref}),
and for $p=2$ they discuss the problem of finding primitive elements for each member of this tower.
While finding the exact orders of given elements in a finite field may be too difficult, it is sometimes 
possible to find lower bounds for the orders of these elements. There is a vast literature on the subject 
and we refer the reader to the paper \cite{BR} by Brochero Martínez and Reis, and the article 
\cite{DMPS} by Dose, Mercuri, Pal, and Stirpe, as well as references therein.

One of the most recent papers studying the tower (\ref{torref}) in the case $p=2$ is due to Popovych \cite{P}.
He gives a lower bound for the order $O(c_i)$ of $c_i$ as well as sufficient conditions for $O(c_i)$, $i\geq 2$, 
to attain its largest possible value, which he proves to be $N_1\cdots N_i$. Under these conditions, he shows that
one also has $O(a_i)=O(c_i)$, $i\geq 2$. As $O(c_0)=3$ and the Fermat numbers $N_i$, $i\geq 1$, are relatively 
prime to $N_0=3$, it follows from (\ref{sizeli}) that for $i\geq 2$, $O(c_i)=N_1\cdots N_i$ if and only if $c_0c_i$ 
is a primitive element of $L_i$. Questions of a similar nature had been addressed earlier by Wiedemann \cite{W}, 
who found other potentially primitive elements for $L_i$.  However, his construction of the fields in the tower 
(\ref{torref}) was different, and an explicit relationship between the two constructions was left open.
Seyfarth and Ranade \cite{SR} show how to translate the problem of the primitivity of Wiedemann's elements
into a specific feature of the characteristic polynomial of certain matrices over $F_2$.  
Recursive constructions of elements of predictably high order in finite fields are useful for the development of  
discrete logarithm ciphers and pseudo-random number generators, and there are added benefits if the 
elements of high order determine normal bases over a ground field (see \cite{PS}).

In this paper we extend Popovych's results from \cite{P} to all primes.  In particular, Theorem 
\ref{mismo} extends and strengthens \cite[Theorem 4]{P} by showing, unconditionally, that 
$O(c_i)=O(a_i)$ for all $i\geq 0$, except only when $p=2$ and $i=1$. In Corollary~\ref{z1}, which 
extends \cite[Theorem 3]{P}, we give a formula for the common order of $c_i$ and $a_i$ as a 
product of the orders of $c_j$ modulo $L_{j-1}^\times$, where $0\leq j\leq i$ if $p$ is odd, and 
$i\geq 2$ and $1\leq j\leq i$ if $p=2$.  Corollary~\ref{z2} provides necessary and sufficient 
conditions for $O(c_i)$ to attain its largest possible value, namely $N_0\cdots N_i$ if $p$ is odd, 
and $N_1\cdots N_i$ if $p=2$ and $i\geq 2$.  As $p-1$ is relatively prime to all $N_i$ (see 
Lemma \ref{relpri}), it follows from (\ref{sizeli}) that if $p$ is odd, $i\geq 0$, and $z$ is a primitive 
element of $\F_p$, then $O(c_i)=N_0\cdots N_i$ if and only if $zc_i$ is a primitive element of $L_i$.
Corollary~\ref{z3} extends \cite[Corollary 2]{P} and gives a lower bound for the order of $c_i$.  

We begin the paper by reviewing Albert's construction of Artin-Schreier field towers in Section \ref{dos}
and establishing their basic properties in Section \ref{secbp}.  Section \ref{main} is devoted to the proof 
of Theorem \ref{mismo} and its consequences.  In Section \ref{normal} and Section \ref{smin} we 
establish two further properties of interest.  Theorem \ref{normalbasis} shows that for all 
$i\geq 0$, the $\mathrm{Gal}(L_i/L_{i-1})$-conjugates of $a_i$ form a normal basis of $L_i$ over 
$L_{i-1}$.  Theorem \ref{x} establishes a connection between the minimal polynomials of $c_i$ and 
$a_{i-1}$ over $\F_p$, $i\geq 1$, which is applied in Corollary \ref{y} to explicitly exhibit the minimal 
polynomial of $c_1$ over $\F_p$.


\section{Construction of the towers}\label{dos}

Let us recall the following result due to Artin and Schreier.

\begin{theorem}\label{as} Let $F$ be a field of prime characteristic
$p$ and take $a\in F$. Then the polynomial $f=X^p-X-a\in F[X]$
is either irreducible or it has a root in $F$.
\end{theorem}

\begin{proof} Let $K$ be a splitting field for $f$ over
$F$, and let $\alpha\in K$ be a root of $f$. As $F$ has
characteristic $p$, we have
$$
(\alpha+1)^p-(\alpha+1)-1=(\alpha^p+1)-(\alpha+1)-1=\alpha^p-\alpha-1=0.
$$
It follows that the $p$ distinct elements
$\alpha,\alpha+1,...,\alpha+(p-1)$ of $K$ are roots of~$f$.
Suppose $f$ has no roots in $F$, so that $\alpha\notin F$. Let
$g\in F[X]$ be the minimal polynomial of $\alpha$ over $F$. The
sum of the roots of $g$ in $K$ belongs to $F$. As $g|f$, this sum  equals
$d\alpha+k$, where $d$ is the degree of $g$ and $k\in F$. If $d<p$
then $\alpha\in F$. As this cannot be, it follows that $d=p$, so
$g=f$ is irreducible.
\end{proof}


\begin{lemma}\label{truno} Let $F$ be a field of prime characteristic
$p$ and take $a\in F$. Suppose $f=X^p-X-a$ has no roots in $F$ and set $K=F[\alpha]$, where $f(\alpha)=0$.
Then the minimal and characteristic polynomial of $\alpha^{p-1}$ over $F$ is $X^p+(X-1)^{p-1}-a^{p-1}-1$.
In particular, $\mathrm{tr}_{K/F}(\alpha^{p-1})=-1$.
\end{lemma}

\begin{proof} From $\alpha^{p}=\alpha+a$ we deduce
$
\alpha^{p-1}=1+\frac{a}{\alpha}.
$
It follows that $F[\alpha]=F[\alpha^{p-1}]$, so
$\alpha$ and $\alpha^{p-1}$ have the same degree
over $F$, namely~$p$ by Theorem \ref{as}. Moreover, $\alpha(\alpha^{p-1}-1)=a$ yields
$$
\alpha^{p-1}(\alpha^{p-1}-1)^{p-1}=a^{p-1},
$$
so the minimal polynomial and characteristic polynomial of $\alpha^{p-1}$ over $F$ is 
$$
X(X-1)^{p-1}-a^{p-1}=(X-1+1)(X-1)^{p-1}-a^{p-1}=X^p+(X-1)^{p-1}-a^{p-1}-1.
$$
\end{proof}

\begin{note} That $\mathrm{tr}_{K/F}(\alpha^{p-1})=-1$ can be seen directly by considering the map
multiplication by $\alpha^{p-1}$, which has the following matrix of trace $-1$ relative to the
basis $1,\alpha,...,\alpha^{p-1}$ of $K$ over $F$:
$$
\begin{bmatrix}
  0 & a & 0 & \cdots & 0 & 0\\
  0 & 1 & a & \cdots & 0  & 0\\
  0 & 0 & 1 & \cdots & 0 & 0\\
  \vdots & \vdots & \vdots & \vdots & \vdots & \vdots \\
  0 & 0 & 0 & \cdots & 1 & a\\
   1 & 0 & 0 & \cdots & 0 & 1\\
\end{bmatrix}.
$$
\end{note}


\begin{lemma}\label{tru} We have
$$ \mathrm{tr}_{L_i/L_{-1}}(a_i)=(-1)^{i+1}, \quad i\geq -1. $$
\end{lemma}

\begin{proof} By induction on $i$. The base case $i=-1$ is trivially true. Suppose that for some $i\geq -1$, we have
\begin{equation}
\label{indh}
\mathrm{tr}_{L_i/L_{-1}}(a_i)=(-1)^{i+1}.
\end{equation}
Since $\mathrm{tr}_{L_i/L_{-1}}(x^p-x)=0$ for all $x\in L_i$, Theorem \ref{as} yields
$$
[L_{i+1}:L_i]=p.
$$
On the other hand, by Lemma \ref{truno}, we have
\begin{equation}
\label{nab}
\mathrm{tr}_{L_{i+1}/L_{i}}(c_{i+1}^{p-1})=-1.
\end{equation}
It follows from the transitivity of the trace, (\ref{indh}), and (\ref{nab}) that
$$
\begin{aligned}
\mathrm{tr}_{L_{i+1}/L_{-1}}(a_{i+1})&=
\mathrm{tr}_{L_{i}/L_{-1}}(\mathrm{tr}_{L_{i+1}/L_{i}}(a_{i+1}))\\
&=
\mathrm{tr}_{L_{i}/L_{-1}}(a_{i}\,\mathrm{tr}_{L_{i+1}/L_{i}}(c_{i+1}^{p-1}))\\&
= \mathrm{tr}_{L_{i}/L_{-1}}(-a_{i})=(-1)^{i+2}.
\end{aligned}
$$
This proves Lemma \ref{tru}. 
\end{proof}

Theorem \ref{as} and Lemma \ref{tru} imply that 
\begin{equation}
\label{dim}
[L_{i+1}:L_{i}]=p, \quad i\geq -1,
\end{equation}
which gives
\begin{equation}
\label{size}
|L_i|=p^{p^{i+1}},\quad i\geq -1.
\end{equation}

\begin{lemma}\label{use} The following identity holds:
$$ c_i^{p^{p^{i}}}=c_i+(-1)^{i},\quad i\geq 0. $$
\end{lemma}

\begin{proof} Consider the equations
$$ 
c_i^p-c_i=a_{i-1},
$$
$$
c_i^{p^2}-c_i^p=a_{i-1}^p,
$$
$$
\vdots
$$
$$
c_i^{p^{p^{i}}}-c_i^{p^{p^{i}-1}}=a_{i-1}^{p^{p^{i}-1}}.
$$
Adding them up and appealing to Lemma \ref{tru}, 
we obtain
\[
c_i^{p^{p^i}}-c_i=\mathrm{tr}_{L_{i-1}/L_{-1}}(a_{i-1})=(-1)^{i}.\qedhere
\]
\end{proof}

It should be noted that the union, say $L$, of the members of the tower
$$
L_{-1}\subset L_0\subset L_1\subset\dots
$$
is the smallest subfield of $E$ that has no extensions of degree $p$. By construction, a basis for $L$ over $\F_p$
is formed by 1 and all elements of the form $c_{n_1}^{e_1}\cdots c_{n_k}^{e_k}$, where $k\geq 1$, $0\leq n_1<\dots<n_k$
and $1\leq e_1,\dots,e_k<p$. We can also construct $L$ as the quotient of the polynomial ring $\F_p[X_{-1},X_0,X_1,\dots]$ by the ideal $I$ generated by 
\begin{equation}\label{ideal}
X_{-1}-1\text{ and }X_{i+1}^p-X_{i+1}-(X_{-1}\cdots X_i)^{p-1},\quad i\geq -1.
\end{equation}
Indeed, there is a ring epimorphism $\F_p[X_{-1},X_0,X_1,\dots]\to L$  mapping $X_j\to c_j$ and containing (\ref{ideal}) in its kernel.
The resulting epimorphism from the quotient ring $\F_p[X_{-1},X_0,X_1,X_2,\dots]/I$ to $L$ is seen to be injective by means of the preceding basis of $L$.

\section{Basic properties}\label{secbp}

In this section we will establish some basic properties of the sequence of numbers $N_i$, $i \ge 0$, for a fixed prime $p$.  These properties were noticed earlier by Popovych \cite{P} in the case $p=2$; for odd primes the proofs are similar.

Recall that
$$
N_i=\frac{p^{p^{i+1}}-1}{p^{p^i}-1},\quad i\geq 0,
$$
so
\begin{equation}\label{prod}
\frac{p^{p^{i+1}}-1}{p-1}=N_0\cdots N_i.
\end{equation}
We can interpret $N_i$ in terms of the norm map $L_i\to  L_{i-1}$, which, according to~(\ref{size}), is given by
$$
x\mapsto x^{1+p^{p^i}+p^{2p^i}+\cdots+p^{(p-1)p^i}}=x^{N_i}.
$$

\begin{lemma}\label{relpri} The numbers $N_i$ and $p^{p^i}-1$ are relatively prime.
\end{lemma}

\begin{proof} Suppose, if possible, that $q$ is a prime factor of
both $N_i$ and $p^{p^i}-1$. Then
$p^{p^{i}}\equiv 1\mod q$ and therefore
$$
0\equiv N_i \equiv 1+p^{p^{i}}+\cdots+p^{p^i(p-1)}\equiv p\mod q,
$$
which implies $0\equiv 1\mod q$, a contradiction.
\end{proof}

\begin{cor}\label{correlpri} Let $d$ be a positive divisor of $N_i$ different from 1. 
Then the order of $p$ modulo $d$ is precisely $p^{i+1}$. Thus, $p^{i+1}|\varphi(d)$. In particular, 
$p^{i+1}|\varphi(N_i)$, and if $q$ is any prime divisor
of $N_i$, then $q\equiv 1\mod p^{i+1}$, so that $d\equiv 1\mod p^{i+1}$.
\end{cor}

\begin{proof} Since $d|N_i$, we have $p^{p^{i+1}}\equiv 1\mod d$.
Therefore the order of $p$ modulo $d$ is a factor of $p^{i+1}$. This order
cannot be a proper divisor of $p^{i+1}$, for otherwise
$p^{p^{i}}\equiv 1\mod d$, which contradicts Lemma \ref{relpri}.
\end{proof}

\begin{cor}\label{correlpri2} If $i\neq j$, then $N_i$ and $N_j$ are relatively prime.
\end{cor}

\begin{proof} Suppose $i>j$. By Lemma \ref{relpri}, $N_i$ is relatively prime to $p^{p^i}-1$
and hence to its factor $p^{p^{j+1}}-1$ and therefore to its divisor $N_j$.
\end{proof}

\begin{lemma}\label{ac} For $i\geq 0$, we have $c_i^{N_i}=a_{i-1}$.
\end{lemma}

\begin{proof} The minimal and characteristic polynomials of $c_i$ over $L_{i-1}$ are equal
to $X^p-X-a_{i-1}$, so $a_{i-1}=N_{L_i/L_{i-1}}(c_i)=c_i^{N_i}$.
\end{proof}

\begin{lemma}\label{ac2} For $i\geq 0$, we have $c_i^{N_0\cdots N_i}=1=a_i^{N_0\cdots N_i}$.
\end{lemma}

\begin{proof} Since $c_{-1}\dots c_i\in L_i$, (\ref{size}) gives
$
(c_{-1}\dots c_i)^{p^{p^{i+1}}-1}=1.
$
But
$N_0\cdots N_i=\frac{p^{p^{i+1}}-1}{p-1}$
and 
$a_i=(c_{-1}\dots c_i)^{p-1}$,
so $a_i^{N_0\cdots N_i}=1$. Lemma \ref{ac} yields $c_0^{N_0}=a_{-1}=1$ as well as
$
c_{i+1}^{N_0\cdots N_iN_{i+1}}=a_i^{N_0\cdots N_i}=1$ for all $i\geq 0$.
\end{proof}

When $p=2$ a sharpening of Lemma \ref{ac} is possible. The next result follows from the fact that finite subgroups of the multiplicative group of a field are cyclic.

\begin{lemma}\label{order} Let $K_1 \subset K_2$ be an extension of fields. Suppose $a\in K_2^\times$ has finite order
and set $m=\min\{n\geq 1\,|\, a^n\in K_1^\times\}$.
Then $O(a)=m\times O(a^m)$.
\end{lemma}


\begin{lemma}\label{ac3} Suppose $p=2$. Then $c_0=a_0$ has order $N_0=3$, $O(c_1)=N_0N_1=15$, $O(a_1)=N_1=5$,
and for $i\geq 2$, we have $c_i^{N_1\cdots N_i}=1=a_i^{N_1\cdots N_i}$.
\end{lemma}

\begin{proof} By Lemma \ref{ac}, $c_0^3=1$. But $c_0^2=c_0+1$, so $c_0\neq 1$. Thus $O(c_0)=O(a_0)=3$.
By (\ref{dim}), $[L_1:L_0]=2$. In particular, $c_1\notin L_0$. On the other hand, Lemma \ref{ac} gives $c_1^5=a_0=c_0\in L_0$.
It follows from Lemma \ref{order} that $O(c_1)=15$. But $a_1=c_1c_0\notin L_0$ and
$a_1^5=c_1^5c_0^5=c_0c_0^2=1$, so $O(a_1)=5$.

We next show that $c_i^{N_1\cdots N_i}=1=a_i^{N_1\cdots N_i}$ for $i\geq 2$. Indeed, by Lemma \ref{ac}, we have
$c_2^{N_1N_2}=a_1^{N_1}=1$, so $a_2^{N_1N_2}=c_2^{N_2N_1}a_1^{N_2N_1}=1$. Suppose that for some $i\geq 2$, we have
$c_i^{N_1\cdots N_i}=1=a_i^{N_1\cdots N_i}$. Then by Lemma \ref{ac},
$$
c_{i+1}^{N_1\cdots N_iN_{i+1}}=a_i^{N_1\cdots N_i}=1,
$$
which implies
\[
a_{i+1}^{N_1\cdots N_iN_{i+1}}=c_{i+1}^{N_1\cdots N_iN_{i+1}} a_i^{N_1\cdots N_iN_{i+1}}=1.\qedhere
\]
\end{proof}

\section{The orders of $c_i$ and $a_i$}\label{main}

\begin{lemma}\label{pripri} For $i\geq 0$, we have $\gcd(N_i-(p^{p^i}-1), N_0\cdots N_i)=1$.
\end{lemma}

\begin{proof} Suppose, if possible,
that $q$ is a prime factor of both numbers. Then $q|N_j$ for some $0\leq j\leq i$.
If $j<i$, then $q|(p^{p^i}-1)$. But $q|(N_i-(p^{p^i}-1))$, so $q|N_i$, against Corollary \ref{correlpri2}.
If $j=i$, then from $q|(N_i-(p^{p^i}-1))$ we infer $q|(p^{p^i}-1)$, against Lemma \ref{relpri}.
\end{proof}

\begin{lemma}\label{pri2} For $i\geq 0$ and  $m\geq 1$, we have 
$$c_i^m\in L_{i-1} \Leftrightarrow c_i^{m(N_i-(p^{p^i}-1))}=a_{i-1}^m.$$
\end{lemma}

\begin{proof} We have $|L_{i-1}|=p^{p^i}$ by (\ref{size}) and $c_i^{N_i}=a_{i-1}$ by Lemma \ref{ac}, so
\[
\begin{aligned}
c_i^m\in L_{i-1} & \Leftrightarrow c_i^{mp^{p^i}}=c_i^m \Leftrightarrow  c_i^{m(p^{p^i}-1)}=1\\
& \Leftrightarrow  c_i^{m(N_i-(N_i-(p^{p^i}-1)))}=1 
\Leftrightarrow  a_{i-1}^m c_i^{-m(N_i-(p^{p^i}-1))}=1\\
& \Leftrightarrow c_i^{m(N_i-(p^{p^i}-1))}=a_{i-1}^m. \qedhere
\end{aligned} 
\]
\end{proof}

\begin{lemma}\label{mci} For $i\geq 0$, let $m$ be the order of $c_i$ modulo $L_{i-1}^\times$.
Then $m$ is relatively prime to the order of $a_{i-1}$.
\end{lemma}

\begin{proof} This is obvious if $i=0$, so we assume $i\geq 1$. We have
$a_{i-1}^{N_0\cdots N_{i-1}}=1$ by Lemma \ref{ac2}, $\gcd(N_0\cdots N_{i-1}, N_i)=1$ by Corollary \ref{correlpri2},
and $m|N_i$ by Lemma~\ref{ac}.
\end{proof}

\begin{prop}\label{c_ia_i} For $i\geq 0$, let $m$ be the order of $c_i$ modulo $L_{i-1}^\times$.
Then
$$
O(c_i)=m \times O(a_{i-1}).
$$
\end{prop}

\begin{proof} By Lemma \ref{ac2}, we have $c_i^{N_0\cdots N_i}=1$, so Lemma \ref{pripri} implies that $c_i$ and $c_i^{N_i-(p^{p^i}-1)}$ generate
the same subgroup. In particular, $m$ is also the order of 
$c_i^{N_i-(p^{p^i}-1)}$ modulo $L_{i-1}^\times$. This, Lemma \ref{order}, Lemma \ref{pri2}, and Lemma \ref{mci} yield
\[
O(c_i)=O(c_i^{N_i-(p^{p^i}-1)})=m\times O(a_{i-1}^m)=m\times O(a_{i-1}).\qedhere
\]
\end{proof}

\begin{theorem}\label{mismo} If $i\geq 0$, then $O(c_i)=O(a_i)$, except only when $p=2$ and $i=1$.
\end{theorem}

\begin{proof} By definition, $a_i=c_i^{p-1}a_{i-1}$. Here $\gcd(p-1,N_0\cdots N_i)=1$ by Lemma \ref{relpri}.
If $i=0$, then Lemma \ref{ac2} yields the desired result. Assume in what follows that $i\geq 1$.
In this case, Lemma \ref{ac2} yields that the orders of $c_i$ and $a_i$ modulo $L_{i-1}^\times$ are identical, say equal to $m$.
Thus by Lemma \ref{order},
\begin{equation}
\label{mix}
O(a_i)=m\times O(a_i^m).
\end{equation}
We have $a_{i}^{N_0\cdots N_{i}}=1$ by Lemma \ref{ac2} and $\gcd(N_i-(p^{p^i}-1), N_0\cdots N_i)=1$ by
Lemma~\ref{pripri}, so
\begin{equation}
\label{mix2}
O(a_i^m)=O(a_i^{m(N_i-(p^{p^i}-1))}).
\end{equation}
Since $a_i^m=c_i^{(p-1)m}a_{i-1}^m$, Lemma \ref{pri2} ensures that
\begin{equation}
\label{mix3}
a_i^{m(N_i-(p^{p^i}-1))}=c_i^{m(p-1)(N_i-(p^{p^i}-1))}a_{i-1}^{m(N_i-(p^{p^i}-1))}=a_{i-1}^{m((p-1)+N_i-(p^{p^i}-1))}.
\end{equation}
It follows from  (\ref{mix2}), (\ref{mix3}), and Lemma \ref{mci} that
\begin{equation}
\label{mix4}
O(a_i^m)=O(a_{i-1}^{(p-1)+N_i-(p^{p^i}-1)})=O(a_{i-1}^{(p-1)+N_i}),
\end{equation}
where the last equality follows from (\ref{prod}) and Lemma \ref{ac2}. 

Suppose, if possible, that $q$ is a prime factor of $\frac{p^{p^i}-1}{p-1}$ and $(p-1)+N_i$.
Then $p^{p^i}\equiv 1\mod q$, so $N_i\equiv p\mod q$. Thus $(p-1)+N_i\equiv 2p-1\mod q$ and therefore $2p\equiv 1\mod q$.
This yields
$$
1\equiv (2p)^{p^i}\equiv 2^{p^i} p^{p^i}\equiv 2^{p^i}\mod q.
$$
Thus the order of 2 modulo $q$ is a positive power of $p$, whence $q\equiv 1\mod p$. In particular, $q>p$, so for any $\ell\geq 2$,
we have $\ell q\geq 2q>2p>2p-1$.  Since $2p\equiv 1\mod q$, this forces $2p-1=q$. But then
$$
1\equiv q\equiv 2p-1\equiv -1\mod p,
$$
whence $p=2$. 

Suppose $p$ is odd. It follows from (\ref{prod}), Lemma \ref{ac2}, and the above analysis that
\begin{equation}
\label{mix5}
O(a_{i-1}^{(p-1)+N_i})=O(a_{i-1}).
\end{equation}
Combining Proposition \ref{c_ia_i} with (\ref{mix}), (\ref{mix4}), and (\ref{mix5}) we obtain the desired result.

Suppose next $p=2$. The cases $i=0$ and $i=1$ were discussed in Lemma~\ref{ac3}. Suppose that $i\geq 2$. Then (\ref{mix4}) becomes
\begin{equation}
\label{mix6}
O(a_i^m)=O(a_{i-1}^{1+N_i-(2^{2^i}-1)})=O(a_{i-1}^{1+(2^{2^i}+1)-(2^{2^i}-1)})=O(a_{i-1}^3).
\end{equation}
Here $a_{i-1}^{N_1\cdots N_{i-1}}=1$ by Lemma \ref{ac3}, and $\gcd(N_1\cdots N_{i-1},3)=1$ by Lemma \ref{correlpri2}, so (\ref{mix6})
implies
\begin{equation}
\label{mix7}
O(a_i^m)=O(a_{i-1}).
\end{equation}
Combining Proposition \ref{c_ia_i} with (\ref{mix}) and (\ref{mix7}) we obtain the sought result.
\end{proof}

\begin{rem} \label{p7rem} {\rm For $i\geq 0$, we write $M_i$ for the order of $c_i$ modulo $L_{i-1}^\times$. Since $[L_i,L_{i-1}]=p$,
it follows that $M_i>1$. By Lemma \ref{ac}, we have $M_i|N_i$, whence  $M_i\equiv 1\mod p^{i+1}$
by Corollary \ref{correlpri}. If $p$ is odd, then $N_i\equiv 1\mod 2$, so in this case 
$M_i\equiv 1\mod 2p^{i+1}$ and therefore $M_i\geq 1+2p^{i+1}$. If $p=2$ and $i\geq 2$, then by the Euler-Lucas theorem,
we have $M_i\equiv 1\mod 2^{i+2}$, which implies $M_i\geq 1+2^{i+2}$. 
If $p=2$, as the Fermat numbers $N_i$, $0\leq i\leq 4$, are prime,
it follows that $M_i=N_i$ for all $0\leq i\leq 4$. } \end{rem}

\begin{cor}
\label{z1} Let $i\geq 0$. If $p$ is odd, we have $O(c_i)=M_0\cdots M_i$.
If $p=2$, then $O(c_0)=N_0$, $O(c_1)=N_0N_1$, $O(c_2)=N_1N_2$, $O(c_3)=N_1N_2N_3$, 
$O(c_4)=N_1N_2N_3N_4$, and for $i\geq 5$, $O(c_i)=N_1N_2N_3N_4M_5\cdots M_i$.
\end{cor}

\begin{proof} If $p$ is odd, the result follows immediately by repeatedly applying Proposition~\ref{c_ia_i}
and Theorem \ref{mismo}.  If $p = 2$, the result follows from Remark \ref{p7rem} and a 
repeated application of Proposition~\ref{c_ia_i}
and Theorem \ref{mismo}.
\end{proof}

\begin{cor}\label{z2} Let $i\geq 0$. If $p$ is odd, then $O(c_i)=N_0\cdots N_i$ if and only if
$M_j=N_j$ for all $0\leq j\leq i$. If $p=2$ and $i\geq 5$, then $O(c_i)=N_1\cdots N_i$ if and only if
$M_j=N_j$ for all $5\leq j\leq i$.
\end{cor}

\begin{cor}\label{z3} Let $i\geq 0$. If $p$ is odd, then $O(c_i)\geq (1+2p)\cdots (1+2p^{i+1})$.
If $p=2$ and $i\geq 5$, then $O(c_i)\geq N_1N_2N_3N_4 (1+2^7)\cdots (1+2^{i+2})$.
\end{cor}

It is shown in \cite[Theorem 5]{P} by means of computer calculations that if $p=2$ then 
$M_i=N_i$ for all $5\leq i\leq 11$, whence $O(c_i)=N_1\cdots N_i$ for all $2\leq i\leq 11$.

\section{A normal basis for $L_i$ over $L_{i-1}$}\label{normal}

Here we show that the Galois conjugates of $a_i$ form a normal  basis for $L_i$ over $L_{i-1}$ for all $i\geq 1$.
Given a field $K$ and $\alpha_1,\dots,\alpha_n\in K$, we write $V(\alpha_1,\dots,\alpha_n)$ for 
the Vandormonde matrix in $M_n(K)$ associated to $(\alpha_1,\dots,\alpha_n)$, that is, with rows equal to $(1,\dots,1),(\alpha_1,\dots,\alpha_n),\dots,(\alpha_1^{n-1},\dots,\alpha_n^{n-1})$.

\begin{theorem}\label{normalbasis}
 For $i\geq 0$, let $\sigma_i$ be the generator of the Galois group of $L_i$ over $L_{i-1}$
given by $x\mapsto x^{p^{p^i}}$. Then $\{a_i,{\sigma_i}(a_i),\dots, {\sigma_i^{p-1}}(a_i)\}$ is a basis for $L_i$ over $L_{i-1}$.
\end{theorem}

\begin{proof} We know that $\{c_i^{p-1},\dots, c_i,1\}$ is a basis of $L_i$ over $L_{i-1}$. Let $M$
be the matrix whose columns are the coordinates of $a_i,{\sigma_i}(a_i),\dots, {\sigma_i^{p-1}}(a_i)$
relative to this basis. We claim that $M$ is invertible. Indeed, by Lemma \ref{use}, we have 
$${\sigma_i}(c_i)=c_i+(-1)^i,{\sigma_i^2}(c_i)=c_i+2(-1)^i,\dots, {\sigma_i^{p-1}}(c_i)=c_i+(p-1)(-1)^i,$$ and therefore
$$a_i=a_{i-1} c_i^{p-1}, {\sigma_i}(a_i)=a_{i-1} (c_i+(-1)^i)^{p-1}, {\sigma_i^2}(a_i)=a_{i-1} (c_i+2(-1)^i)^{p-1},\dots,
$$
$$
{\sigma_i^{p-2}}(a_i)=a_{i-1} (c_i+(p-2)(-1)^i)^{p-1}, {\sigma_i^{p-1}}(a_i)=a_{i-1} (c_i+(p-1)(-1)^i)^{p-1}.
$$
Thus $M=DVS$, where $D=\mathrm{diag}(\binom{p-1}{0}, \binom{p-1}{1},\dots,\binom{p-1}{p-1})$,
$V$ is the Vandermonde matrix $V(0\times (-1)^i, 1\times (-1)^i,\dots, (p-1)\times (-1)^i)$, and $S=\mathrm{diag}(a_{i-1},\dots,a_{i-1})$.

\end{proof}


\section{Minimal polynomials}\label{smin}

\begin{theorem}\label{u} Let $K/F$ be a finite Galois extension with Galois group $G$.
If $f\in F[X]$ is monic and irreducible and $g\in K[X]$ is a monic irreducible factor of $f$ in $K[X]$, then
$$
f=\underset{h\in G g}\Pi h,
$$
where $Gg$ is the orbit of $g$ under the action of $G$.  In particular, we have that $\deg f=\deg h\times |G|/|G_g|$, where $G_g$ is the stabilizer of $g$ in $G$.
\end{theorem}

\begin{proof} It is clear that $\underset{h\in G g}\Pi h\in F[X]$ is a non-constant monic factor of $f$ in $K[X]$ and hence in $F[X]$.
As $f$ is irreducible, $f=\underset{h\in G g}\Pi h$.
\end{proof}

\begin{theorem}\label{v} Let $K/F$ be a finite Galois extension with Galois group $G$. Then for $\alpha\in K$, we have
$$
f_{\alpha,F}=\underset{\beta\in G \alpha}\Pi (X-\beta).
$$
In particular, $\deg f=|G|/|G_\alpha|$ is a factor of $|G|$. Also in particular, if $F=F_q$ and $K=F_{q^n}$, then
$G=\langle T\rangle$, where $T(x)=x^q$ for all $x\in K$, $T$ has order $n$, and if $m$ is the smallest positive integer
such that $T^m(\alpha)=\alpha$, then $m|n$ and 
$$
f_{\alpha,F}=(X-\alpha)\cdots (X-T^{m-1}(\alpha)).
$$
\end{theorem}

\begin{proof} Immediate consequence of Theorem \ref{u}.
\end{proof}

\begin{theorem}\label{w} If $i\geq 0$, then $c_i$ and $a_i$ have degree $p^{i+1}$ over $\F_p$.
\end{theorem}

\begin{proof} By Theorem \ref{v} the degree of $c_i$ over $\F_p$ is a factor of $p^{i+1}$.
If the degree is not $p^{i+1}$, then necessarily $c_i^{p^{p^{i}}}=c_i$,
which contradicts Lemma \ref{use}. Likewise, if the degree of $a_i$ is not $p^{i+1}$, then Lemma \ref{use}
yields
$$
c_i^{p-1}a_{i-1}=a_i=a_i^{p^{p^{i}}}=(c_i+(-1)^i)^{p-1} a_{i-1},
$$
which implies $c_i^{p-1}=(c_i+(-1)^i)^{p-1}$. It follows that $(c_i+(-1)^i)c_i^{-1}\in \F_p$.
But $(c_i+(-1)^i)c_i^{-1}=1+(-1)^i c_i^{-1}$, which is not in $\F_p$ by (\ref{dim}).
\end{proof}

\begin{theorem}\label{x} For $i\geq 1$, we have
$$
f_{c_i,\F_p}(X)=f_{a_{i-1},\F_p}(X^p-X).
$$
In particular, this polynomial is irreducible in $\F_p[X]$.
\end{theorem}

\begin{proof} By definition, $f_{c_i,L_{i-1}}(X)=X^p-X-a_{i-1}$. By Theorem \ref{w},
$a_{i-1}$ has exactly $p^{p^i}$ conjugates over $\F_p$, so Theorem \ref{u} yields
\[
f_{c_i,\F_p}(X)=(X^p-X-a_{i-1})(X^p-X-a_{i-1}^p)\cdots (X^p-X-a_{i-1}^{p^{p^i-1}})=f_{{a_{i-1},\F_p}}(X^p-X).\qedhere
\]
\end{proof}

\begin{cor}\label{y} The minimal polynomial of $c_1$ over $\F_p$ is equal to $$X^{p^2}+(X^p-X-1)^{p-1}-X^p-2.$$
In particular, this polynomial is irreducible in $\F_p[X]$.
\end{cor}

\begin{proof} We have $f_{c_0,\F_p}(X)=X^p-X-1$, so $f_{c_0^{p-1},\F_p}(X)=X^p+(X-1)^{p-1}-2$ by Lemma \ref{truno}.
Thus, by Theorem \ref{x}, 
$$
\begin{aligned}
f_{c_1,\F_p}(X)&=f_{c_0^{p-1},\F_p}(X^p-X)=(X^p-X)^p+(X^p-X-1)^{p-1}-2\\
&=X^{p^2}+(X^p-X-1)^{p-1}-X^p-2.\qedhere
\end{aligned}
$$
\end{proof}

\end{document}